\documentclass[a4paper,10pt]{amsart}
\usepackage{txfonts,dsfont,amsfonts}
\usepackage{mathrsfs}
\usepackage{amsmath,amssymb,amscd,bbm,amsthm}
\usepackage{xy}
\input xy
\input xypic
\xyoption{all}

\numberwithin{equation}{section}

\theoremstyle{plain}
\newtheorem{theorem}{Theorem}[section]
\newtheorem{corollary}[theorem]{Corollary}
\newtheorem{lemma}[theorem]{Lemma}

\newtheorem{proposition}[theorem]{Proposition}

\theoremstyle{definition}
\newtheorem{definition}[theorem]{Definition}

\newtheorem{example}[theorem]{Example}

\theoremstyle{remark}
\newtheorem{remark}[theorem]{Remark}

\newcommand{\N}{\mathbb{N}}
\newcommand{\Z}{\mathbb{Z}}

\DeclareMathOperator{\id}{\rm id }
\DeclareMathOperator{\im}{\rm im}
\DeclareMathOperator{\coker}{\rm coker}

\DeclareMathOperator{\Mod}{\rm Mod}

\DeclareMathOperator{\Hom}{\rm Hom}

\DeclareMathOperator{\Aut}{\rm Aut}
\DeclareMathOperator{\Ext}{\rm Ext}

\DeclareMathOperator{\supp}{\rm supp}

\DeclareMathOperator{\uHom}{\underline{\rm Hom}}

\DeclareMathOperator{\uExt}{\underline{\rm Ext}}

\begin{document}

\title{Skew Calabi-Yau property of normal extensions}

\author{G.-S. Zhou,\; Y. Shen\; and \; D.-M. Lu\ }


\address{\rm Zhou \newline \indent
Ningbo Institute of Technology, Zhejiang University, Ningbo 315100, China
\newline \indent E-mail: 10906045@zju.edu.cn
\newline\newline
\indent Shen \newline
\indent Department of Mathematics, Zhejiang Sci-Tech University, Hangzhou 310018, China \newline \indent E-mail: yuanshen@zstu.edu.cn
\newline\newline
\indent Lu \newline
\indent Department of Mathematics, Zhejiang University, Hangzhou 310027, China
\newline \indent E-mail: dmlu@zju.edu.cn}

\begin{abstract}
We prove that the skew Calabi-Yau property is preserved under normal extension for locally finite positively graded algebras. We also obtain a homological identity which describes the relationship between the Nakayama automorphisms of skew Calabi-Yau locally finite positively graded algebras and their normal extensions. As a preliminary, we show that the Nakayama automorphisms of skew Calabi-Yau algebras always send a regular normal element to a multiple of itself by a unit.  
\end{abstract}




\maketitle




\section{Introduction}

\newtheorem{maintheorem}{\bf{Theorem}}
\renewcommand{\themaintheorem}{\Alph{maintheorem}}
\newtheorem{mainproposition}[maintheorem]{\bf{Proposition}}

Since their introduction \cite{Ginz} in 2006, Calabi-Yau algebras have attracted considerable attention due to their fruitful connections to other topics in mathematics and physics. 
Skew or twisted Calabi-Yau algebras  slightly generalize Calabi-Yau algebras by loosing some of their symmetry. It is well-known that a connected graded algebra is skew Calabi-Yau in the graded sense if and only if it is Artin-Schelter regular \cite{RRZ}. We concern with the general situation, not-necessarily connected, for skew Calabi-Yau property. Interesting examples of such algebras arise naturally as factor algebras of quiver algebras \cite{Bock, BSW}, and as smash products of (connected graded) algebras by Hopf algebras \cite{Le-M}.

Normal extension is a common effective method to construct new algebras with specific ring-theoretic or homological properties. By definition, an algebra $A$ is a normal extension of another algebra $B$ if there exists a regular normal non-invertible element $\hbar\in A$ such that $A/(\hbar) \cong B$. In the graded context we shall require that $\hbar$ is homogeneous and the isomorphism is as graded algebras. This work studies the behavior of the skew Calabi-Yau property  under normal extensions. Take advantage of  Artin-Schelter regularity, it is not hard to see skew Calabi-Yau property is preserved under normal extension in the connected graded case. However, one needs new ideas and new techniques in the general situation. Our approach is independent of the theory of noncommutative dualizing complex  and there is no priori assumption of neotherianness on algebras. 

The main result of this paper (Theorem \ref{Calabi-Yau-normal-extension}) approximately   says that skew Calabi-Yau property is preserved under normal extensions for the larger class of locally finite positively graded algebras. It also gives an identity formula which describes the relationship between the Nakayama automorphisms of skew Calabi-Yau locally finite positively graded algebras and their  normal extensions. This formula was also obtained in \cite[Lemma 1.5]{RRZ} and  \cite[Theorem 2.6]{CKS} under some extra conditions. An interesting preparatory result for the formula is Proposition \ref{Nakayama-auto-behavior}, which  states that the Nakayama automorphisms of skew Calabi-Yau algebras always send a regular normal element to a  multiple of itself by a unit. 

The paper is organized as follows. In section 2, we fix basic notations, recall the definition of skew Calabi-Yau algebras and reveal some features of the Nakayama automorphisms of skew Calabi-Yau algebras. Section 3 is the core of this paper. It  is devoted to study the behavior of the skew Calabi-Yau property under normal extensions. In section 4, we consider the skew Calabi-Yau property for graded algebras endowed with a filtration by homogeneous subspaces from the perspective that their Rees algebras are normal extensions of their associated graded algebras. 
 
Throughout the paper, all algebras and unadorned tensor products $\otimes$  are over a fixed field $\mathbb{K}$. We reserve $G$ to stand for an abelian group. For a ring $A$, a  left (resp. right) $A$-module $M$ and an element $z\in A$, we write $\lambda_z$ (resp. $\rho_z$) for the map  $M\to M$ induced by the scalar multiplication of $z$ on $M$. 

\section{Preliminaries}

Let $A=\bigoplus_{\gamma\in G}A_\gamma$ be a $G$-graded algebra. The set $\supp A=\{\gamma\in G\,|\, A_\gamma\neq0\}$ is called the {\it support} of $A$. If every piece $A_\gamma$ is of finite dimension then $A$ is said to be {\it locally finite}.

We denote by $\Mod^G A$ the category of $G$-graded left $A$-modules with morphisms the $A$-linear maps that preserve degrees. The  category of $G$-graded right $A$-modules (resp. $G$-graded $A$-bimodules) is identified with $\Mod^GA^o$ (resp. $\Mod^GA^e$), where  $A^o$ is the opposite algebra of $A$ and  $A^e$ is the enveloping algebra of $A$, i.e., $A^e=A\otimes A^o$. 
For each $\gamma\in G$, the \emph{$\gamma$-shift functor} $\Sigma_\gamma:\Mod^GA\to \Mod^GA$ sends an object $M$ to $M(\gamma)$, which equals to $M$ as a left $A$-module but with $G$-grading given by $M(\gamma)_{\gamma'}= M_{\gamma+\gamma'}$. The functor $\Sigma_\gamma$ acts as the identity map on morphisms.  For $M,\, N\in \Mod^GA$, we write $$\uHom_{A}^G(M,N):=\bigoplus\nolimits_{\gamma\in G} \Hom_{\Mod^GA}(M,\Sigma_\gamma \, N)$$
and 
$$\uExt_A^{i,G}(M,N):=\bigoplus\nolimits_{\gamma\in G}\Ext_{\Mod^GA}^i(M,\Sigma_\gamma \, N).$$
Note that $\uHom_A^G(M,N)$ and $\uExt_A^{i,G}(M,N)$ become $G$-graded left modules if $M$ is a $G$-graded bimodule, $G$-graded right modules if $N$ is a $G$-graded bimodule, and $G$-graded bimodules if both $M$ and $N$ are $G$-graded bimodules. Similar remark applies to tensor products.

We denote by $\Aut_G(A)$  the group of automorphisms of $A$ that preserve degrees, and by $U(A_0)$ the group of units of $A_0$. Every unit $\xi\in U(A_0)$ gives rise to an automorphism $c_\xi\in \Aut_{G}(A)$ by $a\mapsto \xi a \xi^{-1}$. Such an automorphism is said to be \emph{inner}. For a module $M\in \Mod^GA^e$ and automorphisms $\mu,\, \nu\in \Aut_G(A)$, we define a new module ${}^\mu M^\nu \in \Mod^GA^e$ as follow. It has the same underling vector space and $G$-grading as $M$; the left action and the right action is given by $a*m=\mu(a)\cdot m$ and $m*a=m\cdot \nu(a)$ respectively.
Note that ${}^1A^\mu \cong {}^1A^\nu$ in $\Mod^GA^e$ if and only if $\mu \circ \nu^{-1}$ is inner.

\begin{definition}\label{definition-CY}
We say that a $G$-graded algebra $A$ is \emph{skew $G$-Calabi-Yau of dimension $d$} if
\begin{enumerate}
\item $A$ is \emph{homologically $G$-smooth}, that is,  $A$ has a bounded projective resolution in $\Mod^GA^e$ with each term finitely generated; and
\item there is an automorphism $\mu\in \Aut_G(A)$ and an index $\mathfrak{l}\in G$ such that
\begin{equation*}
\uExt_{A^e}^{i,G} (A,A^e) \cong \left\{
\begin{array}{ll}
0,&i\neq d\\
\Sigma_{\mathfrak{l}}{}^1\!A^\mu, & i=d
\end{array}\right. \quad \text{as $G$-graded $A$-bimodules.}
\end{equation*}
Here, $1=\id_A$ denotes the identity map of $A$.
\end{enumerate}
We call $\mu_{A,G}:=\mu$ the \emph{Nakayama automorphism} of $A$ and $\mathfrak{l}_{A,G}:=\mathfrak{l}$ the \emph{Gorenstein parameter} of $A$. If $\mu_{A,G}$ happens to be the identity map then we say $A$ is \emph{$G$-Calabi-Yau of dimension $d$}.
\end{definition}

\begin{remark}
Due to \cite[Lemma 1.2]{RRZ}, a connected $\Z^r$-graded algebra  $A$ is skew $\Z^r$-Calabi-Yau of dimension $d$ if and only if it is $\Z^r$-AS-regular of dimension $d$, which by definition means that
\begin{enumerate}
\item the global dimensions of $\Mod^{\Z^r}A$ and $\Mod^{\Z^r} A^o$ are both finite, and
\item there is an index $\mathfrak{l}\in G$ such that
\begin{equation*}
\uExt_{A}^{i,G} (\mathbb{K}, A) \cong  \uExt_{A^o}^{i,G} (\mathbb{K}, A) \cong \left\{
\begin{array}{ll}
0,&i\neq d\\
\Sigma_{\mathfrak{l}}\, \mathbb{K}, & i=d
\end{array}\right. \quad \text{as $G$-graded vector spaces.}
\end{equation*}
\end{enumerate}
Here, ``AS'' stands for Artin-Schelter, and connected means that $A_0=\mathbb{K}$ and $\supp A\subseteq \N^r$.
\end{remark}

The next result reveals some features of the Nakayama automorphisms of skew-Calabi-Yau algebras. An element  $z\in A$ is  called \emph{$\sigma$-skew}, where $\sigma$ is an automorphism of $A$, if $z\cdot a=\sigma(a)\cdot z$ for every element $a\in A$ and $\sigma(z)=z$.  Note that every central element is $\id_A$-skew; and if $z$ is homogeneous and regular normal, then there exists a unique  automorphism $\tau_z\in \Aut_G(A)$ so that $z$ is $\tau_z$-skew. In the sequel, let $Z(A)$ be the center of $A$, which is a homogeneous subalgebra of $A$.

\begin{proposition}\label{Nakayama-auto-behavior}
Let $A$ be a $G$-graded algebra and $d\geq0$ an integer. Assume $\uExt_{A^e}^{d,G}(A,A^e)\cong \Sigma_{\mathfrak{l}}{}^1\!A^\mu$ in $\Mod^GA^e$ for some $\mu\in\Aut_G(A)$ and $\mathfrak{l}\in G$. Then there is a map $\Xi : \Aut_G(A) \to U(A_0)$ such that
\begin{enumerate}
\item $\Xi(\sigma\circ\tau) = \Xi(\sigma) \cdot \sigma(\Xi(\tau))$ for every pair of automorphisms $\sigma,\,\tau \in \Aut_G(A)$.
\item $\mu \circ \sigma \circ \mu^{-1} \circ \sigma^{-1}=c_{\Xi(\sigma)}$  for every automorphism $\sigma \in \Aut_G(A)$.
\item $\mu(z) = \Xi(\sigma)\cdot z\,$  for every automorphism $\sigma \in \Aut_G(A)$ and every $\sigma$-skew element $z\in A$.
\end{enumerate}
Consequently, $\mu$ is in the center of $\Aut_G(A)$  when $U(A_0)\subset Z(A)$, and $\mu(z) =z$ when $z$ is central.
\end{proposition}

The proof of this proposition is based on the forthcoming lemma, which connects the group $\Aut_G(A)$ and the $A$-bimodule  $\uExt_{A^e}^{d,G}(A,A^e)$. For modules $M,N\in \Mod^GA^e$ and automorphisms $\mu,\nu\in \Aut_G(A)$, a morphism  $f: M\to N$ in $\Mod^G\mathbb{K}$ is called \emph{$(\mu,\nu)$-linear} if it gives a morphism $f:M\to {}^\mu N^\nu$ in $\Mod^GA^e$. For an automorphism  $\sigma\in \Aut_G(A)$ and an integer  $d\in \Z$, define the $(\sigma,\sigma)$-linear map
$$\Upsilon_{d,\sigma}: \uExt_{A^e}^{d,G}(A,A^e) \to \uExt_{A^e}^{d,G}(A,A^e)$$
to be the composition of following isomorphisms in $\Mod^GA^e$:
\begin{eqnarray*}
\uExt_{A^e}^{d,G}(A,A^e) \xrightarrow{(\sigma^{-1})^*} \uExt_{A^e}^{d,G}({}^{\sigma} A^\sigma,A^e) \xrightarrow{\cong} \uExt_{A^e}^{d,G}(A,{}^{\sigma^{-1}} (A^e)^{\sigma^{-1}}) \xrightarrow{(\sigma\otimes \sigma)_*} {}^\sigma \uExt_{A^e}^{d,G}(A,A^e)^\sigma.
\end{eqnarray*}
Also, we write $\lambda_z$ (resp. $\rho_z$) for the scalar multiplication of $z\in A$ on any left (resp. right) $A$-module.

\begin{lemma}\label{Ext-group-structure}
Let $A$ be a $G$-graded algebra and $d\geq0$ an integer. Then
\begin{enumerate}
\item $\Upsilon_{d,\sigma\circ \tau}= \Upsilon_{d,\sigma}\circ \Upsilon_{d,\tau}$ for every pair of automorphisms $\sigma,\, \tau \in \Aut_G(A)$.
\item $\Upsilon_{d,\sigma} \circ \rho_z =\lambda_z\,$ for every automorphism $\sigma \in \Aut_G(A)$ and every $\sigma$-skew element $z\in A$.
\end{enumerate}
Consequently,  the $A$-bimodule $\uExt_{A^e}^{d,G}(A,A^e)$ is central over $Z(A)$.
\end{lemma}

\begin{proof}
In the proving process, we write $(-,-)$ for $\uExt_{A^e}^{d,G}(-,-)$ temporarily to simplify the notation. 
Part (1)  can be read readily from the following commutative diagram in $\Mod^GA^e$:
$$
\xymatrix{
(A,A^e) \ar[r]^-{(\tau^{-1} \sigma^{-1})^*} \ar[dr]|-{(\tau^{-1})^*} & ({}^{ \sigma   \tau }A^{ \sigma   \tau }, A^e) \ar[rr]^-{\cong} \ar[dr]|-{\cong} & & (A, {}^{(\sigma\tau)^{-1}}(A^e)^{(\sigma\tau)^{-1}}) \ar[r]^-{(\sigma\tau\otimes \sigma\tau)_*} \ar[d]|-{(\tau\otimes\tau)_*}& {}^{\sigma\tau}(A,A^e)^{\sigma\tau}  \\
&({}^{\tau} A^{\tau},A^e) \ar[d]|-{\cong} \ar[u]|-{(\sigma^{-1})^*} & ({}^{\sigma} A^{\sigma}, {}^{\tau^{-1}}(A^e)^{\tau^{-1}}) \ar[ur]|-{\cong} \ar[dr]|-{(\tau\otimes\tau)_*}&{}^\tau(A,{}^{\sigma^{-1}}(A^e)^{\sigma^{-1}})^\tau \ar[ur]|-{(\sigma\otimes\sigma)_*}& \\
&(A,{}^{\tau^{-1}}(A^e)^{\tau^{-1}} ) \ar[ur]|-{(\sigma^{-1})^*} \ar[dr]|-{(\tau\otimes\tau)_*}&&{}^\tau({}^{\sigma} A^{\sigma},A^e)^\tau \ar[u]|-{\cong}&\\
&&{}^\tau(A,A^e)^\tau \ar[ur]|-{(\sigma^{-1})^*}.&&
}
$$
To see Part (2) we may assume further that $z$ is homogeneous of degree $\gamma$.  It is not hard to check that the following  diagram in $\Mod^GA^e$ is commutative:
\begin{eqnarray*}
\xymatrix{
&(A,A^e) \ar[d]|-{\cong}  \ar[r]^-{(\sigma^{-1})^*} & ({}^{\sigma} A^{\sigma},A^e) \ar[d]|-{\cong} \ar[dr]|-{\cong} & \\
(A,A^e) \ar[r]|-{\cong} \ar[ur]|-{\id} \ar[d]|-{\id} \ar[dr]|-{(\lambda_{1\otimes z^o})_*} & ({}^{1}A^{ \sigma^{-1}}, {}^{1}(A^e)^{\sigma^{-1}}) \ar[r]^-{(\sigma^{-1})^*} \ar[d]|-{(\lambda_{1\otimes z^o})^*} & ({}^{\sigma} A^{ 1},{}^{1}(A^e)^{\sigma^{-1}} ) \ar[r]|-{\cong} \ar[d]|-{(\lambda_{z\otimes 1})^*} & (A,{}^{\sigma^{-1}}(A^e)^{\sigma^{-1}}) \ar[dl]|-{(\lambda_{z\otimes 1})_*} \ar[d]|-{(\sigma\otimes\sigma)_*}\\
(A,A^e) \ar[dr]|-{\lambda_z=(\rho_{1\otimes z^o})_*} & \Sigma_{\gamma}(A,{}^{1} (A^e)^{\sigma^{-1}}) \ar[r]^{\id} \ar[d]|-{(1\otimes \sigma)_*}& \Sigma_{\gamma}(A,{}^{1} (A^e)^{\sigma^{-1}})\ar[d]|-{(1\otimes \sigma)_*} & {}^\sigma(A,A^e)^\sigma \ar[dl]|-{\rho_z=(\rho_{z\otimes1})_*} \\
&\Sigma_\gamma {}^\sigma(A,A^e)^1 \ar[r]^{\id}& \Sigma_\gamma {}^\sigma(A,A^e)^1.&
}
\end{eqnarray*}
Thereof, we have $\Upsilon_{d,\sigma} \circ \rho_z =\rho_z\circ \Upsilon_{d,\sigma} = \lambda_z$, as required in Part (2).
Since $\Upsilon_{d,\id_A}$ is the identity map and central elements are $\id_A$-skew, the last statement follows directly from (2).
\end{proof}

\begin{proof}[Proof of Proposition \ref{Nakayama-auto-behavior}]
Fix an isomorphism $\Phi: \Sigma_{-\iota}\uExt_{A^e}^{d,G}(A,A^e)\xrightarrow{\cong} {}^1\!A^\mu$ in $\Mod^GA^e$. For every automorphism $\sigma\in \Aut_G(A)$,  let $\Upsilon_\sigma':{}^1\!A^\mu \to {}^1\!A^\mu$ be the composition of the following bijective maps
$$
{}^1\!A^\mu \xrightarrow{\Phi^{-1}} \Sigma_{-\iota}\uExt_{A^e}^{d,G}(A,A^e)
\xrightarrow{\Sigma_{-\iota} \Upsilon_{d,\sigma}} \Sigma_{-\iota}\uExt_{A^e}^{d,G}(A,A^e) \xrightarrow{\Phi}{}^1\!A^\mu.
$$
Clearly $\Upsilon_\sigma'$ is $(\sigma,\sigma)$-linear, so for every $a\in A$ one has
$$
a\cdot \Upsilon_\sigma'(1_A) = \Upsilon_\sigma'(\sigma^{-1}(a)) = \Upsilon_\sigma'(1_A)\cdot (\mu\circ \sigma \circ \mu^{-1} \circ \sigma^{-1}) (a).
$$
Immediately, it follows that $\Upsilon_\sigma'(1_A)$ is invertible in $A_0$. Define 
$$\Xi:\Aut_G(A) \to U(A_0), \quad \sigma \mapsto \Upsilon_\sigma'(1_A)^{-1}.$$ Then the above equation gives (2). 
By  Lemma \ref{Ext-group-structure} (2), one has 
 $ \Upsilon_{\sigma}' \circ \rho_z=\lambda_z$ and thereof
$$
\mu(z) =\Xi(\sigma)\cdot \Upsilon_{\sigma}'(1_A) \cdot\mu(z)=\Xi(\sigma) \cdot z
$$
for  every $\sigma$-skew element $z\in A$, which is (3). By Lemma \ref{Ext-group-structure} (1), one has
 $\Upsilon_{\sigma\circ \tau}'= \Upsilon_\sigma'\circ \Upsilon_{\tau}'$  and so
$$
\Upsilon_{\sigma\circ \tau}'(1_A) = \sigma(\Upsilon_\tau'(1_A)) \cdot \Upsilon_\sigma'(1_A)
$$
for every pair of automorphisms  $\sigma,\, \tau\in \Aut_G(A)$. Taking inverse one obtains (1).  The last statement is a direct consequence of (2) and (3).
\end{proof}

\begin{remark}
Let $A$ be a connected  $\Z^r$-graded algebra. Suppose $A$ is Artin-Schelter regular (or weaker, Artin-Schelter Gorenstein \cite{JoZ}). We conjecture that in this case the map $\Xi: \Aut_{\Z^r}(A) \to U(A_0)= \mathbb{K}^\times$ in Proposition \ref{Nakayama-auto-behavior} is just  the homological determinant defined by 
J{\o}gensen and Zhang \cite{JoZ}. 
\end{remark}

Next, we focus on the behavior of the skew Calabi-Yau property with respect to regradings.
Let $\varphi:G\to H$  be a homomorphism of abelian groups. We denote by $\varphi^*(A)$ the $H$-graded algebra with the same underlying algebra as  $A$ and the  $H$-grading given by 
$$\varphi^*(A) :=\bigoplus\nolimits_{\delta\in H} \bigoplus\nolimits_{\gamma \in \varphi^{-1}(\delta)} A_\gamma.$$
The map $\varphi$ also induces a functor  $\varphi^*_A: \Mod^GA\to \Mod^H\varphi^*(A)$ as follows. For any  $M\in \Mod^GA$,
$$\varphi^*_A(M)=\bigoplus\nolimits_{\delta\in H} \bigoplus\nolimits_{\gamma \in \varphi^{-1}(\delta)} M_\gamma.$$
The underlying $A$-module structure of  $\varphi^*_A(M)$ is the same as that of $M$. The acting of $\varphi^*_A$ on the morphisms is just the identity maps. Clearly, $\varphi^*_A$ preserves projective resolutions.

\begin{proposition}\label{Calabi-Yau-regrading}
Let $A$ be a $G$-graded algebra and $\varphi:G\to H$ a group homomorphism. If $A$ is skew $G$-Calabi-Yau of dimension $d$ then $\varphi^*(A)$ is skew $H$-Calabi-Yau of dimension $d$; and in this case,
\begin{flalign*}
&& \mathfrak{l}_{\varphi^*(A),H} =\varphi(\mathfrak{l}_{A,G})\quad \text{ and } &  \quad   \mu_{\varphi^*(A),H}=\mu_{A,G}. & 
\end{flalign*}
\end{proposition}

\begin{proof}
Note that $\varphi^*(A)^e = \varphi^*(A^e)$ as $H$-graded algebras, and $A$ admits a projective resolution in $\Mod^GA^e$ with each term finitely generated. 
By standard graded ring theory, we have
 $$
\uExt_{\varphi^*(A)^e}^{i,H} (\varphi^*(A), \varphi^*(A)^e) \cong \varphi_{A^e}^*(\uExt_{A^e}^{i,G}(A,A^e))
$$
in $\Mod^H\varphi^*(A)^e$ for every integer $i\in \Z$. This gives the desired  result.
\end{proof}

\begin{remark}
We don't know whether $\varphi^*(A)$ is skew $H$-Calabi-Yau of dimension $d$ implies $A$ is skew $G$-Calabi-Yau of dimension $d$, even in the special case $G=\Z$ and $H=0$. However, through the AS-regularity, one may tell something other than Proposition  \ref{Calabi-Yau-regrading} on the behavior of the skew Calabi-Yau property with respect to regradings, see \cite[Theorem 2.1.6]{RZ} and \cite[Lemma 1.2]{ZL1} for references.
\end{remark}

\section{Skew Calabi-Yau property under normal extension}

In this section, we study the behavior of the skew Calabi-Yau property under normal extensions. In the connected $\Z^r$-graded case, it is equivalent to investigate the behavior of the AS-regularity under normal extensions, which is easy to handle. However, in the general situation, the AS-regularity plays no role and thereof one needs new ideas and new techniques.

The following notations are employed. We reserve $R$ to stand for a $G$-graded algebra on which a homogeneous regular normal non-invertible element $\hbar$ of degree $\varepsilon$ is specified; we write $\bar{R}:= R/\hbar R$ and consider it $G$-graded naturally; we define $\tau_\hbar\in \Aut_G(R)$  by $\hbar\cdot a ={{\tau_\hbar}}(a)\cdot\hbar$ for all $a\in R$.

Clearly, $\tau_\hbar$ induces an automorphism $\overline{\tau_\hbar} \in \Aut_G(\bar{R})$ because $\tau_\hbar(\hbar) =\hbar$. Due to Proposition \ref{Nakayama-auto-behavior}, if $R$ is skew $G$-Calabi-Yau of dimension $d$ then the Nakayama automorphism $\mu_{R,G} \in \Aut_G(R)$ will also induces an automorphism $\overline{\mu_{R,G}} \in \Aut_G(\bar{R})$. Generally, if $\sigma\in \Aut_G(R)$ sends $\hbar$ to $\xi \cdot \hbar$ for some $\xi\in U(R_0)$, then we define $\overline{\sigma}\in \Aut_G(\bar{R})$ to be the induced automorphism of $\sigma$.

The whole section is devoted to prove the following result.

\begin{theorem}\label{Calabi-Yau-normal-extension}
Suppose that $R$ is locally finite and there is a group homomorphism $p:G\to \Z$ such that $p(\supp R) \subseteq \N$ and $p(\varepsilon) >0$.  Then, if $\bar{R}$ is skew $G$-Calabi-Yau of dimension $d$, it follows that $R$ is skew $G$-Calabi-Yau of dimension $d+1$ with $$\mathfrak{l}_{\bar{R},G} =\mathfrak{l}_{R,G} - \varepsilon \quad \text{ and } \quad \mu_{\bar{R},G} = \overline{\mu_{R,G}}\circ \overline{\tau_\hbar}.$$
\end{theorem}

The proof of this theorem will be addressed  after some preparatory results.  In the following example, we deal with
the quantum affine space as a first toy application of this theorem. 

\begin{example}
Let $n\geq1$ be a positive integer and ${\bf q}=[q_{ij}]_{n\times n}$ be a matrix over $\mathbb{K}$ with entries satisfying $q_{ii}=1$ and $q_{ij}q_{ji}=1$ for all $1\leq i,\, j \leq n$.
The quantum affine $n$-space $A=\mathcal{O}_{\bf q}(\mathbb{K}^n)$ is the algebra generated over  $\mathbb{K}$ by  generators $x_1,\cdots, x_n$ with relations:
$$
x_jx_i= q_{ij}x_ix_j,  \quad 1\leq i,\, j\leq n.
$$
It is known that $A$ is skew Calabi-Yau of dimension $n$ with Nakayama automorphism $\mu_A$ satisfies
$$\mu_A: x_i \mapsto (\Pi_{r=1}^n q_{ri})x_i, \quad i=1,\cdots, n,$$  see \cite[Proposition 4.1]{LWW} and \cite[Example 5.5]{RRZ}.
We regain this fact as follows from a different point of view.  First, introduce a $\Z$-grading on $A$ by putting $\deg(x_s)=1$ for all $s=1,\cdots, n$.  It is easy to see  that  the induced map $\tau_{x_s}\in \Aut_{\Z}(A)$ is given by $x_i\mapsto q_{is}x_i$. Let ${\bf q}'$ (resp. ${\bf q}''$) be the matrix obtained by deleting the last row and the last column (resp. first row and first column) of ${\bf q}$.
Then $ A/(x_n)\cong A':= \mathcal{O}_{{\bf q}'}(\mathbb{K}^{n-1})$ and $A/(x_1)\cong A'':=\mathcal{O}_{{\bf q}''}(\mathbb{K}^{n-1})$ as $\Z$-graded algebras. By induction, we know $A'$, $A''$ are skew $\Z$-Calabi-Yau of dimension $n-1$, and their  Nakayama automorphisms satisfy
\begin{eqnarray*}
\mu_{A'}: &x_i\mapsto (\Pi_{r=1}^{n-1} q_{ri})x_i, &\quad 1\leq i\leq n-1,\\
\mu_{A''}: &x_i \mapsto (\Pi_{r=2}^n q_{ri}) x_i, &\quad 2\leq i\leq n.
\end{eqnarray*}
respectively. Now apply Theorem \ref{Calabi-Yau-normal-extension} and Proposition \ref{Calabi-Yau-regrading}, the desired result follows immediately.
\end{example}

Next, let us proceed to show Theorem \ref{Calabi-Yau-normal-extension}.   Among all, one shall lift the homological smoothness under normal extensions, which  turns out to be
the most troublesome step of our demonstration.
For the fluency of reading,  we postpone the proof of the next proposition to the end of this section.

\begin{proposition}\label{homo-smooth-normal-extension}
Suppose that there is a group homomorphism $p:G\to \Z$ such that  $p(\supp R) \subseteq \N$ and $p(\varepsilon) >0$. Then, 
if $\bar{R}$ is homologically $G$-smooth, it follows that  $R$ is too.
\end{proposition}

In the sequel, one shall bear in mind the following observations: for a $G$-graded algebra $A$, the homogeneous element  $\hbar\otimes 1_{A}$ (resp. $1_A\otimes \hbar^o$) of $R\otimes A$ (resp. $A\otimes R^o$) is  regular, normal and non-invertible. So  results stated for the pair $(R,\hbar)$ may be applied literally to the pair $(R\otimes A, \hbar\otimes1_A) $ (resp. $(A\otimes R^o, 1_A\otimes \hbar^o)$). Also, note that $\tau_{\hbar\otimes 1_A} = \tau_\hbar\otimes \id_A$ and $\tau_{1_A\otimes \hbar^o} = \id_A\otimes \tau_\hbar^{-1}$.

\begin{lemma}[Rees' Lemma]\label{Rees-lemma-regular}
There is a natural isomorphism 
$$\uExt_{\bar{R}}^{i,G}(L,\bar{R}) \cong \Sigma_{-\varepsilon}\uExt_R^{i+1,G}(L,R)^{{\tau_\hbar}}$$ in $\Mod^GR^o$ for every integer $i\in \Z$ and every module $L\in \Mod^G\bar{R}$.
\end{lemma}

\begin{proof}
It follows from the collapsing of the following spectral sequence in $\Mod^GR^o$:
$$\uExt_{\bar{R}}^{p,G}(L, \uExt_{R}^{q,G}(\bar{R}, \Sigma_{-\varepsilon}{}^{{{\tau_\hbar}}^{-1}}\!R^1)) \, \Longrightarrow \, \uExt_{R}^{n,G}(L, \Sigma_{-\varepsilon}{}^{{{\tau_\hbar}}^{-1}}\!R^1)\cong \Sigma_{-\varepsilon}\uExt_{R}^{n,G}(L, R\,)^{{{\tau_\hbar}}}.$$
Note that $\uExt_{R}^{q,G}(\bar{R}, \Sigma_{-\varepsilon}{}^{{{\tau_\hbar}}^{-1}}\!R^1)=0$ for $q\neq 1$ and  $\uExt_R^{1,G}(\bar{R}, \Sigma_{-\varepsilon}{}^{{{\tau_\hbar}}^{-1}}\!R^1)\cong \bar{R}$ in $\Mod^GR\otimes R^o$.
\end{proof}

\begin{lemma}\label{long-exact-sequence-regular}
Assume $M\in \Mod^GR$ is $\hbar$-torsionfree, that is, $\hbar$ annihilates no nonzero element of $M$. Then there is a long exact sequence in $\Mod^GR^{\rm o}$ of the form:
$$
\cdots \to \uExt_{\bar{R}}^{i-1,G}(\bar{R}\otimes_R M, \bar{R})   \to \Sigma_{-\varepsilon}\uExt_R^{i,G}(M,R)^{{{\tau_\hbar}}} \xrightarrow{\rho_{\hbar}} \uExt_R^{i,G}(M,R) \to \uExt_{\bar{R}}^{i,G}(\bar{R}\otimes_R M, \bar{R})  \to   \cdots.
$$
\end{lemma}

\begin{proof}
Consider the following commutative diagram in $\Mod^GR^o$:
{\small$$
\begin{array}{ccccccccccc}
\cdots \!\! & \!\! \to\!\!  & \!\!  \uExt_R^{i,G}(\bar{R}\otimes_R M,R)\!\!  & \!\! \to \!\!  & \!\! \uExt_R^{i,G}(M, R) \!\! & \!\! \xrightarrow{(\lambda_\hbar)^*} \!\!  & \!\!  \uExt_R^{i,G}(\Sigma_{-\varepsilon} {}^{{{\tau_\hbar}}^{-1}}\! M,R) \!\!  & \!\!  \to \!\!  & \!\!  \uExt_R^{i+1,G}(\bar{R}\otimes_R M,R) \!\! & \!\!  \to \!\!   & \!\!  \cdots \\
&&\!\! \downarrow = \!\!& &\!\! \downarrow = \!\!&&\!\! \downarrow \cong \!\!&&\!\! \downarrow = \!\!&& \\
\cdots\!\! & \!\! \to\!\!  & \!\!  \uExt_R^{i,G}(\bar{R}\otimes_R M,R)\!\!  & \!\! \to \!\!  & \!\! \uExt_R^{i,G}(M, R) \!\! & \!\! \xrightarrow{(\lambda_\hbar)_*} \!\!  & \!\!  \uExt_R^{i,G}(M,\Sigma_{\varepsilon} {}^{{{\tau_\hbar}}}\! R) \!\!  & \!\!  \to \!\!  & \!\!  \uExt_R^{i+1,G}(\bar{R}\otimes_R M,R) \!\! & \!\!  \to \!\!   & \!\!  \cdots \\
&&\!\!  \downarrow = \!\! & & \!\!  \downarrow = \!\! && \quad \,\,\, \downarrow({{\tau_\hbar}}^{-1})_*  \!\!&& \!\!  \downarrow =  && \\
\cdots\!\! & \!\! \to\!\!  & \!\!  \uExt_R^{i,G}(\bar{R}\otimes_R M,R)\!\!  & \!\! \to \!\!  & \!\! \uExt_R^{i,G}(M, R) \!\! & \!\! \xrightarrow{(\rho_\hbar)_*} \!\!  & \!\!  \uExt_R^{i,G}(M,\Sigma_{\varepsilon}\, R^{{{\tau_\hbar}}^{-1}}) \!\!  & \!\!  \to \!\!  & \!\!  \uExt_R^{i+1,G}(\bar{R}\otimes_R M,R) \!\! & \!\!  \to \!\!   & \!\!  \cdots\\
&&\!\! \downarrow = \!\!& &\!\! \downarrow = \!\!&&\!\! \downarrow \cong  \!\!&&\!\! \downarrow = \!\!&& \\
\cdots\!\! & \!\! \to\!\!  & \!\!   \uExt_R^{i,G}(\bar{R}\otimes_R M,R)\!\!  & \!\! \to \!\!  & \!\!  \uExt_R^{i,G}(M,R) \!\! & \!\! \xrightarrow{\rho_\hbar} \!\!  & \!\!  \Sigma_{\varepsilon}\uExt_R^{i,G}(M,R)^{{{\tau_\hbar}}^{-1}} \!\!  & \!\!  \to \!\!  & \!\!  \uExt_R^{i+1,G}(\bar{R}\otimes_R M,R) \!\! & \!\!  \to \!\!   & \!\!  \cdots,
\end{array}
$$}where the top row is the long $\uExt$-sequence associated to $0 \to  \Sigma_{-\varepsilon} {}^{{{\tau_\hbar}}^{-1}}\! M \xrightarrow{\lambda_\hbar} M  \to  \bar{R}\otimes_R  M  \to 0$. Apply Lemma \ref{Rees-lemma-regular} to the 1st and 4th term of the bottom row, the result follows.
\end{proof}

\begin{proposition}\label{Nakayama-auto-reduce}
Assume $\uExt_{R^e}^{i,G}(R,R^e)=0$ for $i\neq d$ and $\uExt_{R^e}^{d,G}(R,R^e) \cong \Sigma_{\mathfrak{l}}{}^1\!R^\mu$ for some $\mu\in \Aut_G(R)$ and  $\mathfrak{l}\in G$. Then $\uExt_{\bar{R}^e}^{i,G}(\bar{R},\bar{R}^e)=0$ for $i\neq d-1$ and $\uExt_{\bar{R}^e}^{d-1,G}(\bar{R},\bar{R}^e) \cong \Sigma_{\mathfrak{l}-\varepsilon} {}^1\bar{R}^{\bar{\mu}\circ \bar{{{\tau_\hbar}}}}.$
\end{proposition}

\begin{proof}Apply Lemma  \ref{long-exact-sequence-regular} to  $(R^e, 1_R\otimes \hbar^o)$ with $M=R$, the assumption tells us that  $\uExt_{R\otimes \bar{R}^o}^{i, G}(\bar{R}, R\otimes \bar{R}^o)=0$ for $i\neq d-1,\, d$ and there is an exact sequence in $\Mod^GR^e$ of the form
\begin{eqnarray*}
0\to \uExt_{R\otimes \bar{R}^o}^{d-1,G}(\bar{R},R\otimes \bar{R}^o) \to \Sigma_{\mathfrak{l}-\varepsilon}{}^{\tau_\hbar^{-1}}R^\mu \xrightarrow{\lambda_\hbar} \Sigma_{\mathfrak{l}} {}^1R^\mu \to  \uExt_{R\otimes \bar{R}^o}^{d,G}(\bar{R},R\otimes \bar{R}^o) \to 0.
\end{eqnarray*}
So $\uExt_{R\otimes \bar{R}^o}^{d-1, G}(\bar{R}, R\otimes \bar{R}^o)=0$ and $\uExt_{R\otimes \bar{R}^o}^{d, G}(\bar{R}, R\otimes \bar{R}^o)\cong \Sigma_{\mathfrak{l}} {}^{1} \bar{R}^{\bar{\mu}}$ in $\Mod^GR\otimes \bar{R}^o$. Now, apply Lemma \ref{Rees-lemma-regular} to $(R\otimes \bar{R}^o, \hbar\otimes 1_{\bar{R}^o})$ with $L=\bar{R}$, the result follows directly.
\end{proof}

Now we are ready to prove the main result, namely Theorem \ref{Calabi-Yau-normal-extension}. A module $M\in \Mod^GR$ is  called {\em $\hbar$-discrete} if for any $\gamma\in G$ there is an integer $n\geq0$ such that $(\hbar^{n}\cdot M)_\gamma =0$.   Note that $M=0$ if and only if $\hbar\cdot M=M$ for $\hbar$-discrete modules $M\in \Mod^GR$.  Similar definition and remark apply to modules in $\Mod^GR^o$.The assumptions on $\supp R$ and on $\varepsilon$ in Theorem \ref{Calabi-Yau-normal-extension} is to make sure that every finitely generated $R$-bimodule is $\hbar$-discrete as a left $R$-module and as a right $R$-module.

\begin{proof}[Proof of Theorem \ref{Calabi-Yau-normal-extension}]
Assume $\bar{R}$ is skew $G$-Calabi-Yau of dimension $d$.  By Proposition \ref{homo-smooth-normal-extension}, $R$ is homologically $G$-smooth. 
Let $\mu=\mu_{\bar{R},G}$ and $\mathfrak{l}=\mathfrak{l}_{\bar{R},G}$. Apply Lemma \ref{Rees-lemma-regular} to the pairs $(\bar{R}\otimes R^o, 1\otimes\hbar^o)$ and $(R\otimes \bar{R}^o, \hbar\otimes 1)$ with $L=\bar{R}$, one gets  $\uExt_{\bar{R}\otimes R}^{i,G}(\bar{R},\bar{R}\otimes R)= \uExt_{R\otimes\bar{R}^o}^{i,G}(\bar{R}, R\otimes \bar{R}^o)=0$ for $i\neq d+1$ and $$\uExt_{\bar{R}\otimes R^o}^{d+1,G}(\bar{R},\bar{R}\otimes R^o) \cong \uExt_{R\otimes\bar{R}^o}^{d+1,G}(\bar{R}, R\otimes \bar{R}^o) \cong \Sigma_{\mathfrak{l}+ \varepsilon}\, \bar{R}$$ in $\Mod^GR$ and $\Mod^GR^o$. Then apply Lemma \ref{long-exact-sequence-regular} to $(R^e, \hbar\otimes 1_{R^o})$ and $(R^e, 1_R\otimes \hbar^o)$ with $M=R$, one gets $\uExt_{R^e}^{i,G}(R, R^e)\cdot \hbar = \hbar \cdot \uExt_{R^e}^{i,G}(R, R^e)= \uExt_{R^e}^{i,G}(R, R^e)$ for $i\neq d+1$ and two commutative diagrams
$$
\begin{array}{ccccccccc}
0 &\to & \Sigma_{\mathfrak{l}}\, R^{\tau_\hbar} &\xrightarrow{\rho_\hbar} & \Sigma_{\mathfrak{l}+\varepsilon}\, R &\to & \Sigma_{\mathfrak{l}+\varepsilon}\, \bar{R} &\to & 0\\
&& \downarrow f &&\downarrow f && \downarrow \id&& \\
0 &\to& \Sigma_{-\varepsilon} \uExt_{R^e}^{d+1,G}(R, R^e)^{\tau_\hbar} &\xrightarrow{\rho_\hbar}& \uExt_{R^e}^{d+1,G}(R, R^e) &\to& \Sigma_{\mathfrak{l}+\varepsilon}\ \bar{R} &\to& 0
\end{array}
$$
and
$$
\begin{array}{ccccccccc}
0 &\to & \Sigma_{\mathfrak{l}}\, {}^{\tau_\hbar^{-1}}R &\xrightarrow{\lambda_\hbar} & \Sigma_{\mathfrak{l}+\varepsilon}\,R &\to & \Sigma_{\mathfrak{l}+\varepsilon}\, \bar{R} &\to & 0\\
&& \downarrow g &&\downarrow g&& \downarrow \id&& \\
0 &\to& \Sigma_{-\varepsilon}{}^{\tau_\hbar^{-1}}\uExt_{R^e}^{d+1,G}(R, R^e) &\xrightarrow{\lambda_\hbar}& \uExt_{R^e}^{d+1,G}(R, R^e) &\to& \Sigma_{\mathfrak{l}+\varepsilon} \bar{R}  &\to& 0
\end{array}
$$
in $\Mod^GR^o$ and $\Mod^GR$ respectively with rows exact. Note that $\uExt_{R^e}^{i,G}(R,R^e)$ is $\hbar$-discrete by standard homological methods. So, clearly, $\uExt_{R^e}^{i,G}(R,R^e)=0$ for $i\neq d+1$. Also, $\uExt_{R^e}^{d+1,G}(R,R^e) \cong \Sigma_{\mathfrak{l}+\varepsilon}\, R$ in $\Mod^GR$ and $\Mod^GR^o$ by applying  the well-known Snake Lemma on the above two commutative diagrams. Hence $\uExt_{R^e}^{d+1,G}(R,R^e)\cong \Sigma_{\mathfrak{l}+\varepsilon}\, {}^1R^\nu $ for some $\nu \in \Aut_G(R)$ by \cite[Lemma 2.9]{MM}. Thus, $R$ is skew $G$-Calabi-Yau of dimension $d+1$ with $\mu_{R,G} = \nu$ and  $\mathfrak{l}_{R,G}= \mathfrak{l}+\varepsilon$. Further, ${}^1\bar{R}^{\mu} \cong {}^1\bar{R}^{\overline{\mu_{R,G}}\circ \bar{\tau}_\hbar}$ in $\Mod^G\bar{R}^e$ by Proposition \ref{Nakayama-auto-reduce}.
Note that every inner automorphism of $\bar{R}$ is induced by that of $R$, so $\mu=\bar{c_\xi}\circ \overline{\mu_{R,G}}\circ \bar{\tau}_\hbar$ for some $\xi\in U(R_0)$. Replacing $\mu_{R,G}$ by $c_\xi\circ \mu_{R,G}$, one completes the proof.
\end{proof}

The remaining of the section is devoted to prove Proposition \ref{homo-smooth-normal-extension}. To this end, we introduce two more terminologies.
Let $A$ be a $G$-graded algebra. A module $M\in \Mod^GA$ is said to be {\it (strongly) pseudo-coherent}  if it admits a (bounded) projective resolution in $\Mod^GA$ with each term finitely generated.  So $A$ is homologically $G$-smooth just means that $A$ is strongly pseudo-coherent in $\Mod^GA^e$. The next two lemmas depict the behavior of (strongly) pseudo-coherentness under normal extensions. 

\begin{lemma}\label{pseudo-coherent-lift-torsion} 
For every bounded above complex  $M^*$  in $\Mod^GR$ with each term pseudo-coherent, there exists a surjective quasi-isomorphism $P^*\to M^*$ of complexes in $\Mod^GR$ such that $P^*$ is bounded above with the same right bound as $M^*$ and with each term finitely generated and projective.
Consequently,  (strongly) pseudo-coherent modules  in $\Mod^G\bar{R}$  are also (strongly) pseudo-coherent in  $\Mod^GR$.
\end{lemma}

\begin{proof}
Without lost of generality, we may assume $M^i=0$ for $i>0$.  
Define  $P^i=0$ for $i>0$ and construct by induction on squares $\#_n$ a commutative diagram
$$
\xymatrix{
\ddots& \vdots \ar@{->>}[d] \ar[dr] & \vdots \ar[d] & \vdots \ar[d] & \vdots \ar[d] & \\
\cdots \ar[r] & T^{-3} \ar@{}[ul]|{\#_4} \ar@{}[dr]|{\#_3}\ar[r] \ar[d] & P^{-2} \ar[r] \ar@{->>}[d] \ar[dr] &P^{-1} \ar[r] \ar[d]^{\id} & P^0 \ar[r] \ar[d]^{\id} & 0\\
\cdots \ar[r] & X^{-3} \ar[r]  \ar[d]^{\id}      & T^{-2} \ar@{}[dr]|{\#_2} \ar[r]  \ar[d]      &P^{-1} \ar[r] \ar@{->>}[d] \ar[dr] & P^0 \ar[r] \ar[d]^{\id} & 0\\
\cdots \ar[r] & X^{-3} \ar[r]  \ar[d]^{\id}      & X^{-2} \ar[r] \ar[d]^{\id}       & T^{-1} \ar[r] \ar[d] \ar@{}[dr]|{\#_1} & P^0 \ar[r] \ar@{->>}[d] \ar[dr] & 0\\
\cdots \ar[r] & X^{-3} \ar[r]        & X^{-2} \ar[r]        & X^{-1} \ar[r] & X^0 \ar[r] & 0
}
$$
in $\Mod^GA$ with each rows a complex and each $\#_n$ a pullback square such that $P^{-n}$ is finitely generated and projective and  $\ker (P^{-n}\twoheadrightarrow T^{-n})$ is pseudo-coherent. The required conditions on $\#_n$ assure  $T^{-n}$ is pseudo-coherent, which guarantees this inductive construction. By standard abelian category theory, column morphisms between two adjacent rows form a quasi-isomorphism of complexes. So the composition morphisms $P^{-n} \to T^{-n} \to X^{-n}$ form a surjective quasi-isomorphism $P^* \to X^*$.

It is well known that finitely presented objects of $\Mod^GR$ are closed under taking extensions and direct summands. Since $\bar{R}$ is finitely presented and has projective dimension $1$ in $\Mod^GR$, all finitely generated projective modules in $\Mod^G\bar{R}$ are in particular  pseudo-coherent in $\Mod^GR$.  Now the second statement follows directly from the first one and  the graded version of \cite[Theorem 7.3.5 (i)]{MR} .
\end{proof}

\begin{lemma}\label{pseudo-coherent-lift-torsionfree}
Suppose that $R$ is $\hbar$-discrete as a left $R$-module and $M\in \Mod^GR$ is finitely generated and $\hbar$-torsionfree. Then, if  $\bar{R}\otimes_RM$ admits a  projective resolution $$\bar{P}^*=(\cdots \to \bar{P}^{-2} \xrightarrow{\bar{d}^{-2}} \bar{P}^{-1} \xrightarrow{\bar{d}^{-1}} \bar{P}^0 \to 0 \to \cdots) \quad \xrightarrow{\bar{\eta}} \quad \bar{R} \otimes_RM $$  in $\Mod^G\bar{R}$ with each term finitely generated, it follows that there is a  projective resolution $$P^* = (\cdots \to P^{-2} \xrightarrow{d^{-2}} P^{-1}\xrightarrow{d^{-1}} P^0\to 0\to \cdots) \quad \xrightarrow{\eta} \quad M$$ of $M$ in $\Mod^GR$ with each term finitely generated  and an isomorphism  $F:\bar{R}\otimes_R P^*\xrightarrow{\cong} \bar{P}^*$ of complexes in $\Mod^G\bar{R}$ compatible with augmentations in the sense that  $\bar{R}\otimes_R \eta = \bar{\eta}\circ F$. 
As a consequence, $M$ is (strongly) pseudo-coherent in $\Mod^GR$ if and only if $\bar{R}\otimes_RM$ is too in $\Mod^G\bar{R}$.
\end{lemma}

\begin{proof}
First we claim that for every finitely generated projective module $\bar{Q}\in \Mod^G\bar{R}$, there is a finitely generated projective module $Q\in \Mod^GR$ such that $\bar{Q}\cong \bar{R}\otimes_RQ$ in $\Mod^G\bar{R}$. 
Clearly, $\bar{Q}$ is isomorphic to a direct summand of $\bar{R}\otimes_RL$, where $L=\bigoplus_{i\in I} A (\alpha_i)$ for some finite family $\{\alpha_i\}_{i\in I}$ in $G$. Note that $L$ is $\hbar$-discrete.  Equip $\Hom_{\Mod^GR}(L,L)$ with the linear topology generated by $\{\, U_n\,\}_{n\geq0}$, where  $U_n$ is the subspace consists of all endomorphisms $f:L\to L$ with $f(L)\subseteq \hbar^n\cdot L$. It is easy to check that $\Hom_{\Mod^GR}(L,L)$ is a complete topological algebra and the natural map
\begin{eqnarray*}
\Hom_{\Mod^GR}(L,L) \to \Hom_{\Mod^G\bar{R}}(\bar{R}\otimes_R L,\bar{R}\otimes_RL), \quad  f   \mapsto \bar{R}\otimes_R f,
\end{eqnarray*}
is a surjective homomorphism of algebras whose kernel consists of topologically nilpotent elements. Therefore, by \cite[Lemma 19]{Sj}, there is an idempotent morphism $e:L\to L$ in $\Mod^GR$ such that $\bar{Q}\cong \im(\bar{R}\otimes_Re)$. Consequently, the module $Q:=\im(e)\in \Mod^GR$ fulfills the requirement.

Now assume $\bar{P}^*\xrightarrow{\bar{\eta}} \bar{R}\otimes_RM$ is a projective resolution in $\Mod^G\bar{R}$ with each term finitely generated.
By the above claim, we may construct for every $i\geq0$ a finitely generated projective module $P^{-i}\in \Mod^GR$ and an isomorphism $F^{-i}: \bar{R}\otimes_RP^{-i}\xrightarrow{\cong} \bar{P}^{-i}$ in $\Mod^G\bar{R}$. Note that for every pair of modules $X,Y\in \Mod^GR$ with $X$ projective, the map $\Hom_{\Mod^GR}(X,Y) \to \Hom_{\Mod^G\bar{R}}(\bar{R}\otimes_RX,\bar{R}\otimes_RY)$ given by
$g\mapsto \id_{\bar{R}}\otimes g$ is surjective. So we may choose a map $\eta\in \Hom_{\Mod^GR}(P^0,M)$ so that $ \bar{R}\otimes_R \eta =\bar{\eta}\circ F^0$. It is easy to see  $\bar{R}\otimes_R \coker \eta \cong \coker \bar{\eta}$, so $\coker(\eta)=0$ and hence $\eta$ is surjective. 

We proceed to construct the desired maps $d^{-n}$ for $n\geq1$ inductively. For convenience, we write $P^{1}=M$, $d^{0}=\eta$, $\bar{P}^{1}=\bar{M}$, $\bar{d}^{0}=\bar{\eta}$ and $F^{1}=\id_{\bar{R}\otimes_RM}$. Suppose $ d^{-i}$ has been constructed for $i< n$ so that $P^{-n+1}\xrightarrow{d^{-n+1}} P^{-n+2} \xrightarrow{d^{-n+2}} \cdots \xrightarrow{d^{-1}}P^0\xrightarrow{d^0} P^{1} \to 0$ is exact and $F^{-i+1}\circ (\bar{R}\otimes_R d^{-i}) = \bar{d}^{-i}\circ F^{-i}$. Then
$$0 \to \bar{R}\otimes_R \ker d^{-n+1} \xrightarrow{\bar{R}\otimes_R i_{n}} \bar{R}\otimes_R P^{-n+1} \xrightarrow{\bar{R}\otimes_R d^{-n+1}} \bar{R}\otimes_R P_{-n+2} \to \cdots  $$
is exact, where $i_{n}: \ker d^{-n+1}\to P^{-n+1}$ is the canonical injection. Thus $F^{-n+1}\circ (\bar{R}\otimes_R i_{n})$ induces an isomorphism $F_n': \bar{R}\otimes_R \ker d^{-n+1} \xrightarrow{\cong} \ker \bar{d}^{-n+1} $ in $\Mod^G\bar{R}$. Choose a map $p_n\in \Hom_{\Mod^GR}(P^{-n}, \ker d^{-n+1})$ with $\bar{R}\otimes_R p_n = (F_n')^{-1}\circ \bar{p}_n \circ F^{-n}$, where $\bar{p}_n: \bar{P}^{-n}\to \ker \bar{d}^{-n+1}$ is the canonical projection induced by $\bar{d}^{-n}$. The same reason assuring $\eta$ is surjective also implies that $p_n$ is surjective. Define $d^{-n}= i_n\circ p_n$. Then $P^{-n}\xrightarrow{d^{-n}} P^{-n+1}\xrightarrow{d^{-n+1}}\cdots \xrightarrow{d^{-1}} P^0 \xrightarrow{d^0} P^{1} \to 0 $ is exact and $F^{-n+1}\circ (\bar{R}\otimes_R d^{-n}) = \bar{d}^{-n} \circ F^{-n}$. This finishes the inductive step and proves that the desired morphisms $d^{-1},d^{-2},\cdots,$ exist.

The second statement is  clear by the first one and the following two facts. One is that the functor $\bar{R}\otimes_R-$ turns a projective resolution of $M$ in $\Mod^GR$ to a projective resolution of $\bar{R}\otimes_RM$ in $\Mod^G\bar{R}$. The other one is that $\bar{R}\otimes_R N=0$ if and only if $N=0$ for all finitely generated modules $N\in\Mod^GR$.
\end{proof}

\begin{proof}[Proof of Proposition \ref{homo-smooth-normal-extension}]
Suppose $\bar{R}$ is homologically $G$-smooth. Apply Lemma \ref{pseudo-coherent-lift-torsion} to the pair  $(R\otimes \bar{R}^o, \hbar\otimes 1_{\bar{R}^o})$, it follows that $\bar{R}$ is strongly pseudo-coherent in $\Mod^GR\otimes \bar{R}^o$. Then, apply Lemma \ref{pseudo-coherent-lift-torsionfree} to the pair $(R^e, 1\otimes \hbar^o)$ with $M=R$, it follows that $R$ is homologically $G$-smooth. Here, one shall note $R^e$ is $(1\otimes \hbar^o)$-discrete  and $R$ is $(1\otimes \hbar^o)$-torsionfree as $R^e$-modules. 
\end{proof}

\section{Homo-filtrations}

The theory of filtered rings and filtered modules has been well developed in literatures. In this section, we are interested in graded algebras endowed with a filtration by homogeneous subspaces. It turns out that the skew Calabi-Yau property can be lifted from the associated graded algebra to the original homo-filtered algebra under some mild conditions.

In the sequel, we write $\bar{G} = \Z\times G$. Given a $G$-graded algebra $A$, the Laurent algebra $A[t,t^{-1}]$ is considered to be $\bar{G}$-graded by $\deg(at^i) = (i,\deg(a))$ when $a\in A$ is homogeneous.

By a {\em homo-filtration} on a $G$-graded algebra  $A$ we mean  an increasing sequence $F=\{\,F_nA\,\}_{n\in \Z}$ of homogeneous subspaces of $A$ such that  $1\in F_0A$, $\cup_{n\in Z}F_nA=A$ and $F_mA\cdot F_nA\subseteq F_{m+n}A$ for all $m,\,n\in \Z$. The homo-filtration $F$ is called \emph{positive} if $F_nA=0$ for $n<0$. The $\bar{G}$-graded algebras
$$
R_F(A): =\bigoplus\nolimits_{n\in\Z} F_nA\cdot t^n \ \subseteq \ A[t,t^{-1}]
$$
and
$$
G_F(A):=R_F(A)/(t) \cong
\bigoplus\nolimits_{n\in \Z} F_nA/F_{n-1}A$$
are called respectively the {\em Rees algebra of $A$} and the  {\em associated graded algebra of $A$} with respect to $F$.
Note that $t$ is a regular central homogeneous element of $R_F(A)$ of degree $(1,0)$.

The next two trivial but key lemmas make it possible to lift information from  the associated graded algebra to the homo-filtered algebra in a unified and elegant way.

\begin{lemma}\label{Rees-ring-localization}
Let $A$ be a $G$-graded algebra and $F$ a homo-filtration on $A$. Then, as $\bar{G}$-graded algebras,
\begin{flalign*}
&& R_F(A)[t^{-1}] \cong & A[t,t^{-1}]. & \Box
\end{flalign*}
\end{lemma}

\begin{lemma}\label{Laurent-equivalence}
Let $A$ be a $G$-graded algebra. Then the functor $\Mod^GA\to \Mod^{\bar{G}}A[t,t^{-1}]$ given by $M\mapsto M[t,t^{-1}]$ is an equivalence of abelian categories. 
\end{lemma}

Now we proceed to deal with those properties that we concern.

\begin{proposition}\label{homo-smooth-filtration}
Let $A$ be a $G$-graded algebra and $F$ a positive homo-filtration on $A$.
If $G_F(A)$ is homologically $\bar{G}$-smooth, then $R_F(A)$ is too and $A$ is homologically $G$-smooth.
\end{proposition}

\begin{proof}
Suppose $G_F(A)$ is homologically $\bar{G}$-smooth. Then by Proposition \ref{homo-smooth-normal-extension},  it left to see $A$ is homologically $G$-smooth.
Let $F^e$ be the positive homo-filtration on $A^e$ given by $F^e_nA^e = \sum_{i+j=n} F_iA\otimes F_jA^o$. It is easy to check that $G_{F^e}(A^e) \cong G_F(A)^e$ as $\bar{G}$-graded algebras and
$$G_{F^e}(A^e) \otimes_{R_{F^e}(A^e)} R_F(A) \cong G_{F}(A)$$
in $\Mod^{\bar{G}}G_{F^e}(A^e)$. Then Lemma \ref{pseudo-coherent-lift-torsionfree} tells us that $R_F(A)$ is strongly pseudo-coherent in $\Mod^{\bar{G}}R_{F^e}(A^e)$. It follows that $A[t,t^{-1}]$  is too in $\Mod^{\bar{G}} A^e[t,t^{-1}]$ because of Lemma \ref{Rees-ring-localization} and the isomorphism
$$
A^e[t,t^{-1}] \otimes_{R_{F^e}(A^e)} R_F(A) \cong A[t,t^{-1}]
$$
in $\Mod^{\bar{G}}A^e[t,t^{-1}]$.  Consequently, $A$ is homologically $G$-smooth by Lemma \ref{Laurent-equivalence}. 
\end{proof}

\begin{theorem}\label{Calabi-Yau-filtration-general}
Let $A$ be a $G$-graded algebra and $F$ a positive homo-filtration on $A$ with each layer $F_nA$ locally finite.  Suppose $G_F(A)$ is skew $\bar{G}$-Calabi-Yau of dimension $d$. Then  $R_F(A)$ is  skew $\bar{G}$-Calabi-Yau of dimension $d+1$ and $A$ is skew $G$-Calabi-Yau of dimension $d$. Moreover, one has
\begin{enumerate}
\item  $\mathfrak{l}_{R_F(A),\bar{G}} = \mathfrak{l}_{G_F(A),\bar{G}} + (1,0)$ and  $\mathfrak{l}_{A,G}= \pi(\mathfrak{l}_{G_F(A),\bar{G}})$,
where $\pi:\bar{G}\to G$ assigns $(n,\gamma)$ to $\gamma$.
\item $\mu_{A,G}$ preserves  the homo-filtration $F$ with
 $\mu_{R_F(A),\bar{G}} = R_F(\mu_{A,G})$ and
$\mu_{G_F(A),\bar{G}} = G_F(\mu_{A,G})$.
\end{enumerate}
\end{theorem}

This theorem refines and generalizes  \cite[Corollary 2.16]{Ga} and \cite[Theorem 3.3]{SL}.

\begin{proof}
First note that $t$ is central in $R:=R_F(A)$. By Theorem \ref{Calabi-Yau-normal-extension}, $R$ is skew $\bar{G}$-Calabi-Yau of dimension $d+1$ with $\mu_{R,\bar{G}}$ induces $\mu_{G_F(A), \bar{G}}$ and $\mathfrak{l}_{R,\bar{G}} = \mathfrak{l}_{G_F(A),\bar{G}}+ (1,0)$; and by Proposition \ref{homo-smooth-filtration}, $A$ is homologically $G$-smooth. Also,   $\mu(t) =t$ by Proposition \ref{Nakayama-auto-behavior}.  Let $\nu: A\to A$ be the map defined by $$ F_nA \backslash F_{n-1}A \ni a \quad  \mapsto \quad \mu(at^n) t^{-n} \in F_nA.$$
It is easy to check that $\nu\in \Aut_G(A)$, $\mu_{R,\bar{G}} = R_F(\nu)$ and $\mu_{G_F(A),\bar{G}} = G_F(\nu)$. So it remains to see that $\uExt_{A^e}^{i,G}(A,A^e) =0$ for $i\neq d$ and $\uExt_{A^e}^{d, G}(A,A^e)  \cong \Sigma_{\pi(\mathfrak{l}_{R,\bar{G}})} {}^1 A^\nu$ in $\Mod^GA^e$.

Let $\mu=\mu_{R,\bar{G}}$ and $\mathfrak{l}=\mathfrak{l}_{R,\bar{G}}$.
By Proposition \ref{Calabi-Yau-regrading}, $\pi^*(R)$ is skew $G$-Calabi-Yau of dimension $d+1$ with
$\mu_{\pi^*(R),G}=\mu$ and $\mathfrak{l}_{\pi^*(R),G} = \pi(\mathfrak{l})$. Note that the homogeneous element $\hbar:=t-1\in \pi^*(R)_0$ is  regular and central, and there is a canonical isomorphism of  $G$-graded algebras $A\xrightarrow{\cong} B:=\pi^*(R)/(\hbar)$.  Therefore, by Proposition \ref{Nakayama-auto-reduce}, we have $\uExt_{A^e}^{i,G}(A,A^e) \cong \uExt_{B^e}^{i,G}(B,B^e) =0$ for $i\neq d$ and $$\uExt_{A^e}^{d,G}(A,A^e) \cong \uExt_{B^e}^{d,G}(B,B^e) \cong \Sigma_{\pi(\mathfrak{l})} {}^1B^{\bar{\mu}} \cong \Sigma_{\pi(\mathfrak{l})} {}^1A^\nu$$
in $\Mod^GA^e$, as desired. Here, $\bar{\mu}\in \Aut_G(B)$ is the one induced by $\mu$, $\uExt_{B^e}^{i,G}(B,B^e)$ and ${}^1B^{\bar{\mu}}$ are considered as $G$-graded $A$-bimodules via restricting of scalars along the canonical isomorphism $A\xrightarrow{\cong} B$.
\end{proof}

It is  well-known that Weyl algebras are Calabi-Yau. There are a number of distinct approaches to this conclusion, see \cite{Be,CWWZ,HZ,LWW,SL}. We provide a new one in the following example.
\begin{example}
Let $n\geq1$ be a positive integer. 
The $n$-th Weyl algebra $A_n=A_n(\mathbb{K})$  is the algebra generated over $\mathbb{K}$ by $2n$ generators $x_1,\cdots, x_n,\,y_1,\cdots, y_n$ with relations:
$$
x_iy_i-y_ix_i = 1, \quad x_iy_j-y_jx_i=x_ix_j-x_jx_i = y_iy_j-y_jy_i=0 \quad \text{for} \quad i\neq j.
$$
Introduce a $\Z^n$-grading on $A_n$ by putting $\deg(x_i)=e_i$ and $\deg(y_i)=-e_i$, where $e_i\in \Z^n$ has $1$ at the $i$-th component and $0$ at others. Let  $F=\{\, F_rA_n)\, \}_{r\in \Z}$ be the positive homo-filtration on $A_n$ given  by
$$
F_rA _n= \text{the $\mathbb{K}$-span of all words that have at most $r$ appearances of $y_1,\cdots, y_n$}.
$$
Clearly, all $F_rA_n$ are locally finite  and 
$$
(F_0A_n)_{e_i} = \mathbb{K}x_i \quad \text{and} \quad (F_1A_n)_{-e_i} = \mathbb{K} y_i.
$$
It is not hard to see that $G_F(A) \cong \mathbb{K}[u_1,\cdots, u_n,v_1,\cdots,v_n]$ as $\Z^{n+1}$-graded algebras. Here the polynomial algebra is graded by putting $\deg(u_i)=(0,e_i)$ and $\deg(v_i)=(1,-e_i)$.   Now apply Theorem \ref{Calabi-Yau-filtration-general}  and Proposition \ref{Calabi-Yau-regrading},  
the $n$-th Weyl algebra $A_n$ is Calabi-Yau of dimension $2n$.
\end{example}

We conclude this section with an application of our results to Ore extensions. 

\begin{corollary}\cite[Theorem 0.2] {LWW}
Let $E$ be a $G$-graded algebra contains a homogeneous subalgebra $A$ and a homogeneous  element $x$ of degree $\varepsilon$ satisfying the following two conditions:
\begin{enumerate}
 \item $E$ is a free left $A$-module with bases $\{\, 1,\, x,\, x^2,\cdots\,\}$;
 \item there is an automorphism $\sigma\in \Aut_G(A)$ such that $xa-\sigma(a) x \in A$ for all $a\in A$.
\end{enumerate}
If $A$ is locally finite  and skew $G$-Calabi-Yau of dimension $d$, then $E$ is skew $G$-Calabi-Yau of dimension $d+1$ with
 $\mathfrak{l}_{E,G}  = \mathfrak{l}_{A,G} +\varepsilon$,
 $\mu_{E,G}|_A= \mu_{A,G} \circ \sigma^{-1}$, and
 $\mu_{E,G}(x) = ux +b$ for some $u\in U(A_0)$ and $b\in A_\varepsilon$.
\end{corollary}

\begin{proof}
Let $F=\{F_nE\}_{n\in\Z}$ be the positive homo-filtration on $E$ with each layer $F_nE:=\sum_{i\leq n}Ax^i$. For an element $a\in R_F(E)= \sum_{n\in\Z} (F_nE )t^n$, we denote by 
$\bar{a}$ the coset of $a$ in $G_F(E)=R_F(E)/(t)$. Clearly, $\hbar:= \overline{xt} \in G_F(E)$  is a homogeneous regular normal non-invertible element  of degree $(1,\varepsilon)$.  Let $\iota:G\to \bar{G}$ be the map given by  $\gamma\mapsto (0,\gamma)$ . 
One gets an isomorphism of $\bar{G}$-graded algebras $\iota^*(A) \xrightarrow{\cong} G_F(E)/(\hbar)$
by assigning $a\in A$ to $\bar{a} +(\hbar)$. Assume $A$ is locally finite and skew $G$-Calabi-Yau of dimension $d$. By Proposition \ref{Calabi-Yau-regrading} and Theorem \ref{Calabi-Yau-normal-extension},  the algebra $G_F(E)$ is skew $\bar{G}$-Calabi-Yau of dimension $d+1$ with  $\mathfrak{l}_{G_F(E),\bar{G}} =(1,\mathfrak{l}_{A,G}+\varepsilon)$  and $\mu_{G_F(E), \bar{G}}(\bar{a}) =\overline{\mu_{A,G}(\sigma^{-1}(a))}$ for all $a\in F_0E= A$.  Then Theorem \ref{Calabi-Yau-filtration-general} tells us $E$ is skew $G$-Calabi-Yau of dimension $d+1$ with $\mathfrak{l}_{E,G} = \mathfrak{l}_{A,G}+\varepsilon$, $\mu_{E,G}(a) = \mu_{A,G}(\sigma^{-1} (a)$ for all $a\in A$  and $\mu_{E,G}(x) = ux +b$ for some $u\in A_0$ and $b\in A_{\varepsilon}$.  Since $\mu_{G_F(E),\bar{G}}(\hbar) = G_F(\mu_{E,G})(\hbar) = \overline{\mu_{E,G}(x)} = \bar{u} \hbar$,  it follows that $\bar{u}$ is a unit in $G_F(E)$ by Proposition \ref{Nakayama-auto-behavior}. Thus,  $u$ is a unit in $A$.
\end{proof}

\vskip5mm

\noindent{\it Acknowledgments.} G.-S. Zhou is supported by the NSFC (Grant No. 11601480); Y. Shen is supported by the NSFC (Grant Nos. 11626215, 11701515);  D.-M. Lu is supported by the NSFC (Grant No. 11671351).

\end{document}